\documentclass[12pt,reqno]{amsart}
\usepackage{times}
\usepackage{amssymb}
\usepackage[dvips]{graphicx}

\newtheorem{thm}{Theorem}
\newtheorem{lem}[thm]{Lemma}
\newtheorem{cor}[thm]{Corollary}

\newtheorem{rem}[thm]{Remark}

\newcommand{\e}{{\varepsilon}}
\newcommand{\cp}{C_{\varphi}}
\newcommand{\p}{\varphi}

\newcommand{\rd}{\mathrm{d}}
\newcommand{\sw}{\sigma_w}
\newcommand{\bc}{\mathbb{C}}
\newcommand{\bd}{\mathbb{D}}
\newcommand{\bt}{\mathbb{T}}
\newcommand{\ca}{\mathcal{CA}}
\newcommand{\cl}{\mathcal{L}}
\newcommand{\cb}{\mathcal{B}}
\newcommand{\al}{\mathcal{AL}}
\newcommand{\cm}{\mathcal{CM}}
\newcommand{\pe}{p,\eta}
\newcommand{\db}{\mathrm{dist}_{\cb_{{(3-\eta)}/{2}}}}

\begin{document}

\title {Analytic Campanato Spaces and Their Compositions}
\thanks{Jie Xiao was supported by NSERC of Canada and URP of Memorial University. Cheng Yuan
was supported by NSFC 11226086 of China  and Tianjin Advanced
Education Development Fund 20111005}

\author{Jie Xiao}
\address{Jie Xiao, Department of Mathematics and Statistics, Memorial University, St. John's, NL A1C 5S7, Canada}
\email{jxiao@mun.ca}

\author{Cheng Yuan}
\address{Cheng Yuan, Institute of Mathematics, School of Science, Tianjin University of Technology and Education
Tianjin 300222, China} \email{yuancheng1984@163.com}

\begin{abstract}
This paper is devoted to characterizing the analytic Campanato
spaces $\mathcal{AL}_{p,\eta}$ (including the analytic Morrey
spaces, the analytic John-Nirenberg space, and the analytic
Lipschitz/H\"older spaces) on the complex unit disk $\mathbb D$ in terms of the
M\"obius mappings and the Littlewood-Paley forms, and consequently
their compositions with the analytic self-maps of $\mathbb D$.
\end{abstract}
\keywords{Analytic Campanato spaces, bounded composition operators}

\subjclass[2000]{Primary 30H10, 30H25, 30H30, 30H35, 47A20, 47A25}

\maketitle

\tableofcontents

\section{Introduction}\label{s1}
\setcounter{equation}{0}

\subsection{Background}\label{s11} In partial differential equations and harmonic analysis over the unit circle $\bt$ of the complex plane $\mathbb C$, the
Campanato spaces (named after S. Campanato, cf. \cite{Cam1, Cam2})
$\mathcal L_{p,\eta}(\mathbb T)$ exist as an important family of
spaces that are determined via certain oscillations on $\bt$ and
hence generalize the John-Nirenberg space (named after F. John and
L. Nirenberg, cf. \cite{JN}) $BMO(\mathbb T)$ (bounded mean
oscillation on $\mathbb T$) and the Lipschitz spaces (named after R.
Lipschitz) or the H\"older spaces (named after O. H\"older)
$\mathrm{Lip}_\alpha(\mathbb T)$.

To see this picture clearly, referring to \cite[p. 215]{TO} we adapt
the following definition of the Campanato spaces. For $(p-1,\eta)\in
[0,\infty)\times[0,\infty)$ denote by
$\cl_{p,\eta}:=\cl_{p,\eta}(\bt)$ the $(p,\eta)$-Campanato space of
all functions $f\in L^p(\bt)$ with
$$
\|f\|_{{p,\eta}}=\sup_{I\subseteq \bt}\left(\frac1{|I|^\eta}\int_I
|f(\zeta)-f_I|^p\frac{|\rd \zeta|}{2\pi}\right)^\frac1p<\infty.
$$
In the above and below, $\sup_{I\subseteq \bt}$ means the supremum taken over all subarcs $I$ of $\bt$, and
$$
\begin{cases}
\zeta=e^{it};\\
\rd\zeta=ie^{it}\rd t;\\
|I|=(2\pi)^{-1}\int_I |\rd\zeta|;\\
f_I=(2\pi|I|)^{-1}\int_I f(\zeta)|\rd \zeta|.
\end{cases}
$$
It is easy to get that $\cl_{p,\eta}$ is a Banach space with
the norm
$$
\|f\|_p+\|f\|_{{p,\eta}}=\left(\int_{\bt}|f(\zeta)|^{p}\frac{|\rd\zeta|}{2\pi}\right)^\frac1p+\|f\|_{p,\eta},
$$
and enjoys the following space monotonicity
$$
\begin{cases}
\infty>\eta,\lambda\ge 0\\
\infty>p\ge q\ge 1\\
{(\eta-1)}/{p}\ge{(\lambda-1)}/{q}
\end{cases}
\Longrightarrow \cl_{p,\eta}\subseteq\cl_{q,\lambda}.
$$
Moreover, the following table shows the well-known relationship
among the above-mentioned function spaces (cf. \cite[pp. 65-70]{Gia}
or \cite{Cam1, Cam2, Mey, Sta, Mor, JN}):

\begin{center}
    \begin{tabular}{ | l | p{8cm} |}
    \hline
    Index $(p,\eta)$ & Campanato space $\mathcal{L}_{p,\eta}(\mathbb T)$\\ \hline
    $\eta=0$ & Lebesgue space $L^p(\mathbb T)$\\ \hline
    $\eta\in (0,1)$ & Morrey space $L^{p,1-\eta}(\mathbb T)$\\ \hline
    $\eta=1$ & John-Nirenberg space $BMO(\mathbb T)$ \\ \hline
    $\eta\in (1,1+p]$ & Lipschitz/H\"older space $\mathrm{Lip}_{{(\eta-1)}/{p}}(\mathbb T)$\\ \hline
    $\eta\in (1+p,\infty)$ & Space of constants $\mathbb C$\\ \hline
    \end{tabular}
\end{center}

\subsection{Overview} From the point of view of complex analysis in one variable, it is very natural to study the analytic extensions
$$
\mathcal{AL}_{p,\eta}:=\mathcal{L}_{p,\eta}\cap \hbox{Analytic\
Hardy\ space}\ H^p
$$
of the Campanato spaces on $\bt$ to the open unit disk $\mathbb D$,
where $f\in H^p:=H^p(\mathbb T)$ if and only if $f$ is analytic in
$\mathbb D$ and $\sup_{r\in (0,1)}\|f(r\cdot)\|_{p}<\infty$. Evidently, $\mathcal{AL}_{p,\eta}$ is a Banach space under the norm $\|\cdot\|_p+\|\cdot\|_{p,\eta}$, and has the following structure diagram:

\begin{center}
    \begin{tabular}{ | l | p{8cm} |}
    \hline
    Index $(p,\eta)$ & Analytic Campanato space $\mathcal{AL}_{p,\eta}$\\ \hline
    $\eta=0$ & Analytic Hardy space $H^p$\\ \hline
    $\eta\in (0,1)$ & Analytic Morrey space $AL^{p,1-\eta}$\\ \hline
    $\eta=1$ & Analytic John-Nirenberg space $BMOA$ \\ \hline
    $\eta\in (1,1+p]$ & Analytic Lipschitz/H\"older space $\mathrm{ALip}_{{(\eta-1)}/{p}}$\\ \hline
    $\eta\in (1+p,\infty)$ & Space of constants $\mathbb C$\\ \hline
    $p\ge q\ \&\ \frac{\eta-1}{p}\ge\frac{\lambda-1}{q}$ & $\mathcal{AL}_{p,\eta}\subseteq\mathcal{AL}_{q,\lambda}$\\ \hline
    \end{tabular}
\end{center}
As far as we know, the following papers \cite{WX, X1, XX, CFO} have touched some aspects of these analytic Campanato spaces. Thus, such an area deserves in the highest degree to be developed further. In this paper, we will establish the following assertion:

\begin{thm}\label{t11} For $w\in\bar{\mathbb D}=\mathbb D\cup\mathbb T$ let $z\mapsto \sigma_w(z)=(w-z)(1-\bar{w}z)^{-1}$ be a M\"obius self-map of $\bar{\mathbb D}$.

\item{\rm(i)} For $0<\eta<2<1+p<\infty$, an analytic function $f$ on $\mathbb D$ belongs to $\mathcal{AL}_{p,\eta}$ if and only if
$$
\|f\|_{\mathcal{AL}_{p,\eta,\ast}}:=\sup_{w\in\mathbb
D}(1-|w|^2)^\frac{1-\eta}{p}\|f\circ\sigma_w-f(w)\|_p<\infty
$$
if and only if
$$
\|f\|_{\mathcal{AL}_{p,\eta,\star}}:=\sup_{I\subseteq\bt}\left( \frac1{|I|^\eta} \int_I
\left(\int_{1-|I|}^1 |f'(r \zeta)|^2 (1-r)\rd
r\right)^\frac{p}{2}\,\frac{|\rd\zeta|}{2\pi}\right)^\frac{1}{p}<\infty.
$$

\item{\rm(ii)} For $0<\eta\le 1+p<\infty$, $1<p<\infty$ and $\alpha=(p+1-\eta)/p$, $\mathcal{AL}_{p,\eta}$ is contained in the Bloch type space
$$
\mathcal{B}_\alpha=\Big\{f:\ f\ \hbox{is\ analytic\ in}\ \mathbb D\
\hbox{with}\ \|f\|_{\mathcal
B_{\alpha}}:=\sup_{w\in\bd}(1-|w|^2)^\alpha|f'(w)|<\infty\Big\}.
$$
In particular, for $p=2$ and $f\in\mathcal{B}_\alpha$, one has
$$
\inf_{g\in\mathcal{AL}_{p,\eta}}\|f-g\|_{\mathcal
B_{\alpha}}=\inf\Big\{\e :\ \chi_{\Omega_\e
(f)}(1-|z|^2)^{\eta-2}\rd A(z)\ \ \hbox{is\ in}\
\mathcal{CM}_\eta\Big\},
$$
where
$$
\begin{cases}
\Omega_\e (f)=\{z\in\bd: (1-|z|^2)^{{(1+p-\eta)}/{p}}|f'(z)|\ge\e\};\\
\chi_{\Omega_\e(f)}\ \ \hbox{is\ the\ characteristic\ function\ of}\ \Omega_\e(f);\\
\rd A(z)\ \ \hbox{is\ the\ normalized\ area\ measure\ on}\ \ \bd;\\
\mathcal{CM}_\eta\ \ \hbox{is\ the\ class\ of\ all}\
\eta\hbox{-Carleson\ measures\ on}\ \mathbb D.
\end{cases}
$$

\item{\rm(iii)} For $0<\eta,\lambda<2=q\le p<\infty$, an analytic self-map $\p $ of $\mathbb D$, and the analytic composition $C_\p f=f\circ\p$, one has:
$$
\|C_\p f\|_{\mathcal{AL}_{q,\lambda,\ast}}\lesssim\|f\|_{\mathcal{AL}{p,\eta,\ast}}\quad\forall\quad f\in \mathcal{AL}_{p,\eta}
$$
if and only if
$$
\sup_{w\in
\bd}\frac{(1-|w|^2)^{(1-\lambda)/q}}{(1-|\p(w)|^2)^{(1-\eta)/p}}
\|\sigma_{\p(w)}\circ\p\circ\sw\|_{q}<\infty.
$$

\item{\rm(iv)} For $0<\eta<1+p<\infty$, $0<\alpha, p-1<\infty$, an analytic self-map $\p $ of
$\mathbb D$, and the analytic composition $C_\p  f=f\circ\p$, one has:
$$
\|C_\p f\|_{\mathcal{B}_\alpha}\lesssim\|f\|_{\mathcal{AL}_{p,\eta,\star}}\quad\forall\quad f\in \mathcal{AL}_{p,\eta}
$$
respectively
$$
\|C_\p f\|_{\mathcal{AL}_{p,\eta,\star}}\lesssim\|f\|_{\mathcal{B}_\alpha}\quad\forall\quad f\in \mathcal{B}_\alpha
$$
if and only if
$$
\sup_{w\in\mathbb D}\frac{(1-|w|^2)^\alpha|\p'(w)|}{(1-|\p
(w)|^2)^\frac{p+1-\eta}{p}}<\infty
$$
respectively
$$
\sup_{I\subseteq\bt}\left(\frac1{|I|^\eta} \int_I
\left(\int_{1-|I|}^1 |\p'(r \zeta)|^2 (1-r)^{1-2\alpha}\rd
r\right)^\frac{p}{2}\,\frac{|\rd\zeta|}{2\pi}\right)^\frac{1}{p}<\infty.
$$
\end{thm}

\begin{proof}[Notation] As we go along with proving Theorem \ref{t11} in \S \ref{s2}-\S\ref{s3}-\S\ref{s4}, we will introduce the required notation. But here, we only use
$U\lesssim V$ (or $V\gtrsim U$) as $U\le c V$ for a positive
constant $c$, and moreover write $U\approx V$ for both $U\lesssim V$ and
$V\lesssim U$.
\end{proof}

\section{Harmonic extensions of Campanato spaces}\label{s2}
\setcounter{equation}{0}

\subsection{Representation via M\"obius transforms}\label{s21}

For $f\in L^1(\bt)$ and $z\in\mathbb D$, let $P_z(\zeta)=(1-|z|^2)/
|\zeta-z|^{2}$ be the Poisson kernel at $z$ and $Pf$ be the Poisson
extension of $f$, that is
$$
Pf(z)=\frac1{2\pi}\int_\bt f(\zeta)P_z(\zeta)\,|\rd
\zeta|=\frac1{2\pi}\int_0^{2\pi}\Re(\frac{e^{i\theta}+z}{e^{i\theta}-z})f(e^{i\theta})\rd
\theta.
$$
As an extension of Theorem 3.2.1 in \cite{X1}, we have the following
M\"obius type representation of a Campanato space, whence reaching
the first part of Theorem \ref{t11} (i).

\begin{thm}\label{t21} Let $0<\eta<2\le 1+p<\infty$.

\item{\rm(i)} For an $L^p$ function $f$, $f\in \cl_{\pe}$ if and only if
\begin{equation*}
\|f\|_{\pe,*}:=\sup_{w\in\bd}\left((1-|w|^2)^{1-\eta}\int_\bt
\left|f\circ\sigma_w(\zeta)- Pf(w)\right|^p\,|\rd
\zeta|\right)^\frac1p<\infty.
\end{equation*}

\item{\rm(ii)} For an $H^p$ function $f$, $f\in \mathcal{AL}_{\pe}$ if and only if $\|f\|_{\mathcal{AL}_{p,\eta,\ast}}<\infty$.
\end{thm}

\begin{proof} (i) Suppose $f\in \cl_{\pe}$. For each nonzero
$w\in\bd$, let $I_w$ be the subarc of $\bt$ with center $w/|w|$ and
length $1-|w|$; and for $w=0$, let $I_w=\bt$. Moreover, for
$j=0,1,...,n-1$, let $J_j=2^j I_w$, where $n$ is the smallest
natural number such that $2^n|I_w|\ge1$. Obviously, we may assume
$J_n=\bt$. Under this circumstance, we have that for a given point
$w\in\bd$,
\begin{eqnarray*}
&&
\int_\bt|f(\zeta)-Pf(w)|^p\frac{1-|w|^2}{|\zeta-w|^2}\,|\rd \zeta|\\
&&\lesssim\int_\bt|f(\zeta)-f_{I_w}|^p\frac{1-|w|^2}{|\zeta-w|^2}\,|\rd \zeta|\\
&&\lesssim\left(\int_{J_0}+\sum_{j=0}^{n-1}\int_{J_{j+1}\setminus
J_j}
\right)|f(\zeta)-f_{I_w}|^p\frac{1-|w|^2}{|\zeta-w|^2}\,|\rd \zeta|\\
&&\lesssim
(1-|w|^2)^{-1}\left(\int_{J_0}+\sum_{j=0}^{n-1}2^{-2j}\int_{J_{j+1}\setminus
J_{j}}
\right)|f(\zeta)-f_{I_w}|^p\,|\rd \zeta|\\
&&\lesssim\frac1{|J_0|}\int_{J_0}|f(\zeta)-f_{I_w}|^p |\rd \zeta|+
\sum_{j=0}^{n-1}\frac{2^{-j}}{|J_{j+1}|}\int_{J_{j+1}\setminus J_j}
|f(\zeta)-f_{I_w}|^p\,|\rd \zeta|.\\
\end{eqnarray*}
Using H\"{o}lder's inequality and $f\in \cl_{\pe}$, we get
\begin{align*}
&|f_{J_{j+1}}-f_{J_j}|\\
&\lesssim\frac1{|J_{j+1}|}\int_{J_{j+1}}|f(\zeta)-f_{I_w}||\rd \zeta|\\
&\lesssim\left(\frac2{|J_{j+1}|}\int_{J_{j+1}}|f(\zeta)-f_{I_w}|^p|\rd \zeta|\right)^\frac1{p}\\
&\lesssim |J_{j+1}|^{\frac{\eta-1}{p}}\|f\|_{\pe}\\
&\lesssim 2^{(\eta-1)(j+1)/p}(1-|w|)^{\frac{\eta-1}{p}}\|f\|_{\pe},
\end{align*}
whence reaching
\begin{align*}
&|f_{J_{j+1}}-f_{J_0}|\\
&\lesssim |f_{J_{j+1}}-f_{J_j}|+\cdots+|f_{J_{1}}-f_{J_0}|\\
&\lesssim j(1+2^{j(\eta-1)/p})(1-|w|)^{\frac{\eta-1}{p}}\|f\|_{\pe}.
\end{align*}
Accordingly,
\begin{align*}
&\frac{1}{|J_{j+1}|}\int_{J_{j+1}}|f(\zeta)-f_{J_0}|^p\,|\rd \zeta|\\
&\le\left(\left(\frac{1}{|J_{j+1}|}\int_{J_{j+1}}
|f(\zeta)-f_{J_{j+1}}|^p\,|\rd \zeta|\right)^{1/p}+ |f_{J_{j+1}}-f_{J_0}| \right)^p\\
&\lesssim j^p(1+2^{j(\eta-1)})(1-|w|)^{{\eta-1}}\|f\|^p_{\pe}.
\end{align*}
Putting the above estimates together, we obtain
\begin{align*}
&\int_\bt
\left|f\circ\sigma_w(\zeta)-Pf(w)\right|^p\,|\rd \zeta|\\
&\approx\int_\bt
\left|f(\zeta)- P f(w)\right|^p\frac{1-|w|^2}{|\zeta-w|^2}\,|\rd \zeta|\\
&\lesssim(1-|w|)^{{\eta-1}}\|f\|^p_{\pe}\sum_{j\ge 0}
2^{-j}{j^p(1+2^{j(\eta-1)})}.
\end{align*}
Since $0<\eta<2\le 1+p$, we have $\|f\|_{\pe,*}\lesssim
\|f\|_{\pe}.$

Conversely, if $\|f\|_{\pe,*}<\infty$, then for any subarc
$I\subseteq\bt$, we take such a $w\in\bd$ that $w/|w|$ is the center
of $I$ and $|w|=1-|I|$, to get
$$
|\zeta-w|\lesssim 1-|w|^2\quad\forall\quad \zeta\in I,
$$
and then
\begin{align*}
&\frac{1}{|I|^\eta}\int_{I}
|f(\zeta)-f_{I}|^p\,|\rd \zeta|\\
&\lesssim\frac{1}{|I|^\eta}\int_{I}
|f(\zeta)-Pf(w)|^p\,|\rd \zeta|\\
&\lesssim(1-|w|^2)^{1-\eta}\int_{I}
|f(\zeta)-Pf(w)|^p\frac{1-|w|^2}{|\zeta-w|^2}\,|\rd \zeta|\\
&\lesssim\|f\|^p_{\pe,*},
\end{align*}
as desired.

(ii) This follows from (i) and the fact that that $Pf=f$ holds for
all $f\in H^p$.
\end{proof}

Given $z\in\mathbb D$ and $1<p<\infty$. Upon writing the Szeg\"{o}
projection $S$ of $f\in L^p(\bt)$:
\begin{align*}Sf(z)=\frac{1}{2\pi}\int_0^{2\pi} \frac{f(e^{i\theta})}{1-ze^{-i\theta}}
\rd\theta,
\end{align*}
and $\tilde f$ the conjugate operator of $f$:
$$
\begin{cases}
\tilde
f(z)=\frac1{2\pi}\int_0^{2\pi}\Im(\frac{e^{i\theta}+z}{e^{i\theta}-z})f(e^{i\theta})\rd
\theta\\
\hbox{with}\\
\tilde f(e^{i\theta})=\lim_{r\to1^-}\tilde f(r e^{i\theta}),
\end{cases}
$$
one has
\begin{align*}
i\tilde f(z)+Pf(z)=2Sf(z)-Pf(0)
\end{align*}

Closely related to Peetre's \cite[Theorem 1.1]{Pee} (cf. \cite{PeeJFA}), the following result extends the
$BMO$ space case (see \cite[p. 272]{Zhu}) and the Lipschitz space
case (Privalov's theorem, see \cite[p. 214]{TO}).

\begin{cor} Let $0<\eta<2<1+p<\infty$. The Szeg\"{o} projection $S$ maps $\cl_{\pe}$ boundedly
onto $\al_{\pe}$. Equivalently, the conjugate operator is bounded on
$\cl_{\pe}$.
\end{cor}
\begin{proof} For $f\in\cl_{\pe}\subset L^p$, write the
Szeg\"o projection of $f$ as $S(f)=u+iv$, where $u$ and $v$ are the
real and imaginary parts respectively. Then
\begin{align*}
(Sf)\circ\sw=u\circ\sw+iv\circ\sw.
\end{align*}
This means that $u$ is the Poisson extension of $f$. According to
the well known result: $L^p$ boundedness of $S:L^p\to H^p$ when
$1<p<\infty$ (see, e.g. Theorem 9.6 in \cite{Zhu}), we have
\begin{align*}\|Sf\circ\sw-Sf(w)\|_p\lesssim\|u\circ\sw-u(w)\|_p.
\end{align*}
Thus
\begin{align*}(1-|w|^2)^{1-\eta}\|Sf\circ\sw-Sf(w)\|_p^p
\lesssim(1-|w|^2)^{1-\eta}\|u\circ\sw-u(w)\|_p^p.
\end{align*}
Note that $u|_{\bt}=f$, then $Sf\in\al_{\pe}$ thanks to Theorem \ref{t11} (i). So one gets
\begin{align*}
&(1-|w|^2)^{1-\eta}\|v\circ\sw-v(w)\|_p^p\\
&\lesssim(1-|w|^2)^{1-\eta}\|Sf\circ\sw-Sf(w)\|_p^p\\
&\lesssim\|u\|_{\pe,*}^p\\
&\approx \|f\|_{\pe,*}^p.
\end{align*}
This completes the proof.
\end{proof}

\subsection{Littlewood-Paley type characters}\label{s22}

For $\eta\in(0,1]$ and $1\le p<\infty$, let $p'=p/(p-1)$ with
convention $1'=1/(1-1)=\infty$. A function $a$ on $\bt$ is called
$(p',\eta)$-atom provided that there is a subarc $I\subseteq\bt$
such that
$$
\begin{cases}
\mathrm{supp}\,  a\subset I;\\
\int_\bt a(\zeta)|\rd \zeta|=0;\\
\|a\|_{p'}\le |I|^{-\eta/p}.
\end{cases}
$$

Also, let $\mathcal{H}_{p',\eta}:=\mathcal{H}_{p',\eta}(\bt)$ be the
space of all functions
$$
g(\zeta)=\sum_{j=0}^\infty \lambda_j a_j(\zeta)\quad\hbox{in\ the\
sense\ of\ distribution},
$$
where $\lambda_j\in\bc$, $a_j$ is a $(p',\eta)$-atom and
$\sum_{j=0}^\infty |\lambda_j|<\infty$.

The following result is known; see also \cite[Proposition 5]{Z} and
its references including \cite{FS}.

\begin{lem}\label{l22}
\item{\rm(i)} If $0<\eta<1<p<\infty$, then $\cl_{\pe}$ is isomorphic to the dual space
$\mathcal{H}^*_{p',\eta}$ under the pairing
$$
\int_\bt f(\zeta)\overline{g(\zeta)}
\frac{|\rd\zeta|}{2\pi}\quad\forall\quad (f,g)\in
\cl_{\pe}\times\mathcal{H}_{p',\eta}.
$$
\item{\rm(ii)} If $\eta=1\le p<\infty$, then $\cl_{\pe}=BMO$ is isomorphic to the dual space
$\mathcal{H}^*_{\infty,1}$ under the pairing
$$
\int_\bt f(\zeta)\overline{g(\zeta)} \frac{|\rd
\zeta|}{2\pi}\quad\forall\quad (f,g)\in
BMO\times\mathcal{H}_{\infty,1}.
$$
\end{lem}

Inspired by \cite{DXY, Tor, OT, FJN} we have the following
equivalent characterization of the Campanato-Morrey spaces, thereby
arriving at the second part of Theorem \ref{t11} (i).

\begin{thm}\label{t22} Let
$$
\begin{cases}
\nabla=\big(\partial_z+\partial_{\bar{z}}, i(\partial_z-\partial_{\bar{z}})\big);\\
1<p<\infty;\\
\eta\in(0,1+p).
\end{cases}
$$

\item{\rm(i)} For an $L^p$ function $f$, $f\in\cl_{\pe}$ if and
only if
\begin{equation*}
\|f\|_{p,\eta,\star}:=\sup_{I\subseteq\bt}\left( \frac1{|I|^\eta}
\int_I \left(\int_{1-|I|}^1 |\nabla Pf(r \zeta)|^2 (1-r^2)\rd
r\right)^\frac{p}{2} \frac{|\rd\zeta|}{2\pi}\right)^\frac1p<\infty.
\end{equation*}

\item{\rm(ii)} For an $H^p$ function
$f$, $f\in \al_{\pe}$ if and only if $\|f\|_{\mathcal{AL}_{p,\eta,\star}}<\infty$.
\end{thm}
\begin{proof} It suffices to verify (i). We first prove the ``only if" part. If
$f\in\cl_{\pe}$, then any given subarc $I\subseteq\bt$ is used to
decompose $f$ via

$$
f=f_I+(f-f_I)\chi_{2I}+(f-f_I)\chi_{\bt\setminus 2I}=f_1+f_2+f_3,
$$
where $\chi_J$ is the characteristic function of $J\subseteq\bt$,
$2I\subseteq\bt$ is the concentric double of $I$.

In sake of convenience, suppose that $u_j=Pf_j$ is the Poisson
extension of $f_j$. Then, we control $|\nabla u_j|$.

For $u_1$, we have $\nabla u_1=0$ since $u_1$ is a constant.

For $u_2$, we ultilize the Littlewood-Paley estimate for $L^p$ to
derive
\begin{align*}
&\int_I \left(\int_{1-|I|}^1 |\nabla u_2(r \zeta)|^2 (1-r^2)
\rd r\right)^\frac{p}{2}\,|\rd\zeta|\\
&\lesssim\int_{\bt} \left(\int_0^1 |\nabla u_2(r \zeta)|^2
(1-r^2)\rd r\right)^\frac{p}{2}\,|\rd\zeta|\\
&\lesssim\|f_2\|^p_p\\
&\lesssim \|f\|_{p,\eta}^p|I|^\eta.
\end{align*}

For $u_3$, we have that if $z=r e^{i\theta}$ then
\begin{align*}
&|\nabla u_3(z)|\approx\left|\int_\bt
f_3(\zeta)\left(\nabla\frac{1-|z|^2}{|\zeta-z|^2}\right)|\rd\zeta|\right|\\
&\lesssim\int_0^{2\pi}\frac{|f_3(e^{it})|}{(1-r)^2+(\theta-t)^2}
\rd t\\
&\approx\int_{\bt\setminus 2I}\frac{|f(e^{it})-f_I|}{(1-r)^2+(\theta-t)^2}
\rd t\\
&\lesssim\int_{|\theta-t|>|I|}\frac{|f(e^{it})-f_I|}{(\theta-t)^2}
\rd t\\
&\lesssim\sum_{j\ge 0}\int_{2^j|I|<|\theta-t|\le2^{j+1}|I|}
\frac{|f(e^{it})-f_I|}{(2^j|I|)^2}
\rd t\\
&\lesssim\sum_{j\ge 0}\frac{\int_{|\theta-t|\le2^{j+1}|I|}
{|f(e^{it})-f_{2^{j+1}I}|} \rd
t}{(2^j|I|)^2}+\sum_{j\ge 0}\frac{|f_{2^{j+1}I}-f_I|}{2^j|I|}\\
&\lesssim\sum_{j\ge
0}\frac{1}{(2^j|I|)^2}\left(\int_{|\theta-t|\le2^{j+1}|I|}
{|f(e^{it})-f_{2^{j+1}I}|^p} \rd
t\right)^{\frac{1}{p}}\left(2^{j+1}|I|\right)^{1-\frac{1}{p}}\\
&\ +\sum_{j\ge 0}\frac{1}{2^j|I|}
\left(\frac1{|I|}\int_{2^{j+1}I}|f(e^{it})-f_{2^{j+1}I}|^p \rd
t\right)^{\frac{1}{p}}\\
&\lesssim\|f\|_{p,\eta}\sum_{j\ge 0}\left({2^j|I|}\right)^{\frac{\eta-1}{p}-1}\\
&\lesssim\|f\|_{p,\eta}|I|^{\frac{\eta-1}{p}-1}.
\end{align*}
Clearly, the H\"older's inequality has been used in the above
estimation. As a further result, we get

\begin{align*}
& \int_I \left(\int_{1-|I|}^1 |\nabla u_3(r \zeta)|^2 (1-r^2)\rd
r\right)^\frac{p}{2} |\rd\zeta|\\
&\lesssim|I|^{(\frac{\eta-1}{p}-1)p} \|f\|_{p,\eta}^p \int_I
\left(\int_{1-|I|}^1
  (1-r^2)\rd
r\right)^\frac{p}{2} |\rd\zeta|\\
&\lesssim\|f\|_{p,\eta}^p|I|^\eta.
\end{align*}

Putting the previous estimates for $|\nabla u_j|$ together, we
obtain
$$
\|f\|_{p,\eta,\star}\lesssim\|f\|_{p,\eta}<\infty.
$$

Conversely, suppose $\|f\|_{p,\eta,\star}<\infty$. To prove
$f\in\mathcal{L}_{p,\eta}$, we are required to consider the three
cases:
$$
\begin{cases}
0<\eta<1;\\
\eta=1;\\
1<\eta<1+p.
\end{cases}
$$
The first and second cases are based on a dual-related formula
below:

\begin{eqnarray*}
\int_\bt f(\zeta)\overline{g(\zeta)} \frac{|\rd
\zeta|}{2\pi}=Pf(0)\overline{Pg(0)} +\int_\bd \nabla P
f(z)\cdot\overline{\nabla P g(z)}\log\frac1{|z|}\rd A(z).
\end{eqnarray*}
However, the third case follows from a mean value estimate for the
harmonic functions.

{\it Case 1}: $0<\eta<1$. For simplicity, set $Pf=F$ and $Pg=G$. It
suffices to handle the convergence of the pair:
$$
\langle F,G\rangle:=\int_\bd \nabla F(z)\cdot\overline{\nabla
G(z)}(1-|z|^2)\rd A(z)\quad\forall\quad f\in\cl_{\pe}\ \&\
(p',\eta)-\hbox{atom}\ g.
$$
To do so, thanks to Lemma \ref{l22} (i) we may assume that
$f\in\cl_{\pe}$ and $g$ is a $(p',\eta)$-atom generated by a subarc
$I\subseteq\bt$ with $|I|\ll 1$. Then, we estimate

$$
|\langle F,G\rangle|\lesssim\int_\bt\int_0^1 |\nabla
F(r\zeta)|{|\nabla G(r\zeta)|} (1-r)\rd r |\rd\zeta|\approx
\iint_{S(2I)}+\iint_{\mathbb D\setminus S(2I)},
$$
where
$$
\begin{cases}
S(2I)=\{z=r\zeta: r\in [1-2|I|,1)\ \&\ \zeta\in 2I\};\\
\iint_{S(2I)}=\int_{2I}\int_{1-2|I|}^1 |\nabla F(r\zeta)|{|\nabla
G(r\zeta)|}
(1-r)\rd r |\rd\zeta|;\\
\iint_{\mathbb D\setminus S(2I)}=\iint_{\bd\setminus S(2I)} |\nabla
F(r\zeta)|{|\nabla G(r\zeta)|} (1-r)\rd r |\rd\zeta|.
\end{cases}
$$

For $\iint_{S(2I)}$, we use the H\"older's inequality twice and the
$L^{p^\prime}$-bound of the Littlewood-Paley $G$-function to obtain
\begin{align*}
&\iint_{S(2I)}\lesssim\int_{2I}\frac{\left(\int_{1-2|I|}^1 |\nabla
F(r\zeta)|^2 (1-r)\rd r\right)^{\frac{1}{2}}}{\left(\int_{1-2|I|}^1
|\nabla
G(r\zeta)|^2 (1-r)\rd r\right)^{-\frac{1}{2}}}|\rd\zeta|\\
&\lesssim\left(\int_{2I}\left(\int_{1-2|I|}^1 |\nabla F(r\zeta)|^2
(1-r)\rd r\right)^{\frac{p}{2}} |\rd\zeta|\right)^{\frac1{p}}\\
&\ \times\left(\int_{2I}\left(\int_{1-2|I|}^1 |\nabla G(r\zeta)|^2
(1-r)\rd
r\right)^{\frac{p'}{2}}|\rd\zeta|\right)^{\frac1{p'}}\\
&\lesssim\left(\int_{2I}\left(\int_{1-2|I|}^1 |\nabla F(r\zeta)|^2
(1-r)\rd r\right)^{\frac{p}{2}}
|\rd\zeta|\right)^{\frac1{p}}\|g\|_{L^{p^\prime}}\\
&\lesssim\|f\|_{p,\eta,\star} |I|^{\frac{\eta}{p}}\|g\|_{L^{p^\prime}}\\
&\lesssim\|f\|_{p,\eta,\star}.
\end{align*}

For $\iint_{\mathbb D\setminus S(2I)}$, recalling that $g$ is a
$(p',\eta)$-atom associated to $I$ centered at
$\zeta_0=e^{i\theta_0}$, we get
\begin{align*}
G(z)=\int_\bt
g(\zeta)\left(\frac{1-|z|^2}{|\zeta-z|^2}-\frac{1-|z|^2}{|\zeta_0-z|^2}\right)\frac{|\rd
\zeta|}{2\pi},
\end{align*}
whence deriving
\begin{align*}
&\nabla G(z)=\int_\bt
g(\zeta)\nabla\left(\frac{1-|z|^2}{|\zeta-z|^2}-\frac{1-|z|^2}{|\zeta_0-z|^2}\right)\frac{|\rd
\zeta|}{2\pi}\\
&=\int_I
g(\zeta)\nabla\left(\frac{1-|z|^2}{|\zeta-z|^2}-\frac{1-|z|^2}{|\zeta_0-z|^2}\right)\frac{|\rd
\zeta|}{2\pi}.
\end{align*}

To proceed, let
$$
[\zeta_0,\zeta]:=\{\phi(t):=\zeta_0+(\zeta-\zeta_0)t\in\mathbb D:\
t\in [0,1]\}
$$
be the line segment connecting $\zeta\in I$ and $\zeta_0\in I$.
Since $|\zeta|=|\zeta_0|=1$, we can use the mean value theorem for
derivatives to obtain some $t_0\in [0,1]$ such that
\begin{align*}
&\left|\partial_z\left(\frac{1-|z|^2}{|\zeta-z|^2}-
\frac{1-|z|^2}{|\zeta_0-z|^2}
\right)\right|\\
&=\left|\partial_z\Big(\frac{1-|z|^2}{|\zeta-z|^2}\Big)-\partial_z\Big(\frac{1-|z
|^2}{|\zeta_0-z|^2}\Big)\right|\\
&=\left|\partial_t\left(\partial_z\Big(\frac{1-|z|^2}{|\phi(t)-z|^2}\Big)\right)\right|_{t=t_0}
|\zeta-\zeta_0|\\
&=\left|\partial_z\left(\partial_t\Big(\frac{1-|z|^2}{|\phi(t)-z|^2}\Big)\right)\right|_{t=t_0}
|\zeta-\zeta_0|.
\end{align*}
A direct calculation gives
\begin{align*}
  \partial_t\Big(\frac{1-|z|^2}{|\phi(t)-z|^2}\Big)
  =2\Re\phi'(t)\overline{(\phi(t)-z)}
  \frac{1-|z|^2}{|\phi(t)-z|^4}.
\end{align*}
Thus
\begin{align*}
&\left|\partial_z\left(\frac{1-|z|^2}{|\zeta-z|^2}-
\frac{1-|z|^2}{|\zeta_0-z|^2} \right)\right|\\
& =\left|\partial_z\left(2\Re\phi'(t)\overline{(\phi(t)-z)}
  \frac{1-|z|^2}{|\phi(t)-z|^4}\right)\right|_{t=t_0}
|\zeta-\zeta_0|.
\end{align*}
Noticing
$$
\begin{cases}
\phi'(t)=\zeta-\zeta_0;\\
2\Re\phi'(t)\overline{(\phi(t)-z)}=\phi'(t)\overline{(\phi(t)-z)}+
\overline{\phi'(t)}(\phi(t)-z),
\end{cases}
$$
we compute
\begin{align*}
  &\partial_z\left( \phi'(t)(\overline{\phi(t)}-\overline{z})
  \frac{1-|z|^2}{|\phi(t)-z|^4}\right)\\
  &=\frac{\phi'(t)(\overline{\phi(t)}-\overline{z})}{|\phi(t)-z|^6}
  \big(-\overline{z}|\phi(t)-z|^2+2(\overline{\phi(t)}-\overline{z})\big),
\end{align*}
and
\begin{align*}
  &\partial_{z}\left( \overline{\phi'(t)}({\phi(t)}-{z})
  \frac{1-|z|^2}{|\phi(t)-z|^4}\right)\\
  &  =\frac{\overline{\phi'(t)}({\phi(t)}-{z})}{|\phi(t)-z|^6}
  \Big(-\overline{z}|\phi(t)-z|^2+2(\overline{\phi(t)}-\overline{z})
  \Big)-\frac{\overline{\phi'(t)}(1-|z|^2)}{|\phi(t)-z|^4},
\end{align*}
whence getting
\begin{align*}
\left|\partial_z\left(\partial_t\Big(\frac{1-|z|^2}{|\phi(t)-z|^2}\Big)\right)\right|
\lesssim \frac{|\phi'(t)|}{|\phi(t)-z|^4}.
\end{align*}
If $z\in S(2^{j+1}I)\setminus S(2^jI)$ (where $j=1,2,3,...$) and $t\in [0,1]$, then
$$
|\phi(t)-z|\gtrsim 2^j |I|\gtrsim 2^j
|\zeta-\zeta_0|,
$$
and hence
$$
 \left|\partial_z\left(\frac{1-|z|^2}{|\zeta-z|^2}-\frac{1-|z|^2}{|\zeta_0-z|^2}\right)
\right|\lesssim
\frac{|\zeta-\zeta_0|^2}{|\phi(t)-z|^4}\lesssim\frac1{2^{4j}|I|^2} .
$$
Similarly, we have
$$
\left|\partial_{\bar{z}}\left(\frac{1-|z|^2}{|\zeta-z|^2}-
\frac{1-|z|^2}{|\zeta_0-z|^2} \right)\right|
 \lesssim\frac1{2^{4j}|I|^2}.
$$
Hence
\begin{align*}
&\left|\nabla\left(\frac{1-|z|^2}{|\zeta-z|^2}-\frac{1-|z|^2}{|\zeta_0-z|^2}\right)
\right|\\
&\lesssim\left|\partial_z\left(\frac{1-|z|^2}{|\zeta-z|^2}-\frac{1-|z|^2}{|\zeta_0-z|^2}\right)\right|+
\left|\partial_{\bar z}\left(\frac{1-|z|^2}{|\zeta-z|^2}-\frac{1-|z|^2}{|\zeta_0-z|^2}\right)
\right|\\
&\lesssim\frac{1}{2^{4j}|I|^2}.
\end{align*}
Therefore, we employ $\|g\|_{p'}\lesssim|I|^{-\frac{\eta}{p}}$ and
H\"older's inequality to obtain
$$
\left|\nabla G(z)\right|
\lesssim\frac{\int_I|g(
\zeta)||\rd\zeta|}{2^{4j}|I|^2}
\lesssim\frac{|I|^{\frac{1-\eta}{p}-2}}{2^{4j}}\quad\forall\quad
z\in S(2^{j+1}I)\setminus S(2^jI).
$$

Upon writing
$$
\iint_{\mathbb D\setminus S(2I)}=\sum_{j\ge
1}\int_{2^jI}\int_{1-2^{j+1}|I|}^{1-2^{j}|I|}+\sum_{j\ge
1}\int_{2^{j+1}\setminus 2^jI}\int_{1-2^{j+1}|I|}^1,
$$
we have two types of estimates as follows. The first type is:
\begin{align*}
&\sum_{j\ge 1}\int_{2^jI}\int_{1-2^{j+1}|I|}^{1-2^{j}|I|}\\
&\lesssim\sum_{j\ge 1}\int_{2^{j}I}\bigg(
\big(\int_{1-2^{j+1}|I|}^{1-2^j
|I|}|\nabla F(r\zeta)|^2(1-r)\rd r\big)^{\frac1{2}} \\
&\ \times \big(\int_{1-2^{j+1}|I|}^{1-2^j |I|}|\nabla
G(r\zeta)|^2(1-r)\rd r\big)^{\frac1{2}}\bigg) |\rd\zeta|\\
&\lesssim\sum_{j\ge 1}\left(\int_{2^{j}I}\bigg(
\int_{1-2^{j+1}|I|}^{1-2^j
|I|}|\nabla F(r\zeta)|^2(1-r)\rd r\bigg)^{\frac{p}{2}}|\rd\zeta|\right)^{\frac1{p}} \\
&\ \times \left(\int_{2^{j}I}\bigg( \int_{1-2^{j+1}|I|}^{1-2^j
|I|}|\nabla G(r\zeta)|^2(1-r)\rd
r\bigg)^{\frac{p'}{2}}\,|\rd\zeta|\right)^{\frac1{p'}}\\
&\lesssim\|f\|_{p,\eta,\star}\sum_{j\ge 1}|2^{j}I|^{\eta/p}
\left(\int_{2^{j}I}\bigg( \int_{1-2^{j+1}|I|}^{1-2^j |I|}
(\frac{|I|^{\frac{1-\eta}{p}-2}}{2^{4j}})^2 (1-r)\rd
r\bigg)^{\frac{p'}{2}}\,|\rd\zeta|\right)^{\frac1{p'}} \\
&\lesssim
\|f\|_{p,\eta,\star}|I|^{\frac{\eta}{p}+(\frac{1-\eta}{p}-2)}
\sum_{j\ge 1}2^{(j+1)\frac{\eta}{p}-4j} \left(\int_{2^{j}I}\bigg(
\int_{1-2^{j+1}|I|}^{1}
(1-r)\rd r\bigg)^{\frac{p'}{2}}\,|\rd\zeta|\right)^{\frac1{p'}} \\
&\lesssim\|f\|_{p,\eta,\star}|I|^{\frac{1}{p}-2} \sum_{j\ge 1}
2^{(j+1)\frac{\eta}{p}-4j} \left(2^{j+1}|I|(2^{j+1}|I|)^{p'}\right)^{\frac1{p'}} \\
&\lesssim\|f\|_{p,\eta,\star}|I|^{\frac{1}{p}-2+1+\frac1{p'}}\sum_{j\ge
1}
2^{(j+1)\frac{\eta}{p}-4j+j+1+\frac{j+1}{p'}}\\
&\lesssim\|f\|_{p,\eta,\star}\sum_{j\ge 1} 2^{j(\frac{\eta-1}{p}-2)}\\
&\lesssim\|f\|_{p,\eta,\star}.
\end{align*}
In the above we have used the H\"older's inequality twice, the
estimate for $|\nabla G|$ on $S(2^{j+1}I)\setminus S(2^jI)$ and the
assumption $0<\eta<1$.

The second type is:
$$
\sum_{j\ge 1}\int_{2^{j+1}\setminus
2^jI}\int_{1-2^{j+1}|I|}^{1}\lesssim\|f\|_{p,\eta,\star}\sum_{j\ge
1} 2^{j(\frac{\eta-1}{p}-2)},
$$
which may be verfied in a way similar to the first type. As a
result, we have
$$
\iint_{\mathbb D\setminus S(2I)}\lesssim \|f\|_{p,\eta,\star},
$$
whence reaching
$$
|\langle F,G\rangle|\lesssim\|f\|_{p,\eta,\star}.
$$
This in turn implies $f\in \mathcal{L}_{p,\eta}$.

{\it Case 2}: $\eta=1$. Regarding this situation, we apply
$(\infty,1)$-atoms (due to Lemma \ref{l22} (ii)) to the previous
discussion, and then obtain
$$
|\langle F,G\rangle|\lesssim\|f\|_{p,1,\star},
$$
which yields $f\in BMO$.

{\it Case 3}: $1<\eta< 1+p$. Suppose $\|f\|_{p,\eta,\star}<\infty$
and set $u=Pf$. For a given $z=r e^{i\theta}\in\bd$, let $S(I)$ be
the Carleson box with $|I|=(1-|z|)$ and $e^{i\theta}$ be the center
of $I$. Then

\begin{align*}|I|^\eta\|f\|^p_{p,\eta,\star}&\gtrsim\int_I\Big(\int_{1-|I|}^1|\nabla u(r
e^{i\theta})|^2(1-r)\rd r\Big)^{\frac{p}{2}}\rd \theta.
\end{align*}
Since $u$ is harmonic in $\mathbb D$, $\nabla u$ is harmonic
overthere as well. Consequently, there exists some number $\rho>0$
such that $B(z,\frac{1-|z|}{\rho})\subset S(I)$, where
$B(z,\frac{1-|z|}{\rho})$ is the open disk with center $z$ and
radius $\frac{1-|z|}{\rho}$. This implies the following mean value
inequality:
\begin{align*}
(1-|z|)|\nabla u(z)| \lesssim
{(1-|z|)^{-2}}\int_{B(z,\frac{1-|z|}{\rho})} |\nabla u(w)|(1-|w|)
\rd A(w).
\end{align*}
So, it follows from the H\"older's inequality that
\begin{align*}
&(1-|z|)|\nabla u(z)|\\
&\lesssim {(1-|z|)^{-2}}\int_{S(I)} |\nabla u(w)|(1-|w|) \rd
A(w)\\
&\lesssim (1-|z|)^{-2}\int_I\left(\int_{1-|I|}^1 |\nabla u(r
e^{i\theta})|^2(1-|r|) \rd r\right)^{\frac1{2}}
\left(\int_{1-|I|}^1(1-r)\rd r\right)^{\frac1{2}}\rd\theta\\
&\lesssim {(1-|z|)^{-1}}\int_I\left(\int_{1-|I|}^1 |\nabla u(r
e^{i\theta})|^2(1-|r|) \rd r\right)^{\frac1{2}} \rd\theta\\
&\lesssim \frac{1}{|I|}\left(\int_I\big(\int_{1-|I|}^1 |\nabla u(r
e^{i\theta})|^2(1-|r|) \rd r\big)^{\frac{p}{2}}
\rd\theta\right)^{\frac1{p}} |I|^{1-\frac1{p}}\\
&\lesssim |I|^\frac{\eta-1}{p}\|f\|_{p,\eta,\star}\\
&\lesssim (1-|z|)^\frac{\eta-1}{p}\|f\|_{p,\eta,\star}.
\end{align*}
This means that
$$
(1-|z|)^{\frac{1-\eta}{p}+1}|\nabla
u(z)|\lesssim\|f\|_{p,\eta,\star}
$$
and thus
$$
f\in\mathrm{Lip}_{\frac{\eta-1}{p}}=\cl_{\pe}\quad\hbox{under}\quad
1<\eta< 1+p.
$$
\end{proof}

\section{Analytic Campanato spaces in Bloch type spaces}\label{s3}
\setcounter{equation}{0}

\subsection{Generalized Carleson measures}\label{s31} For $\eta>0$, we write $\mathcal{CM}_\eta$ for the class of all
$\eta$-Carleson measures on $\mathbb D$. Recall that a nonnegative Borel measure
$\mu$ on $\bd$ is called an $\eta$-Carleson measure provided
$$
\|\mu
\|_{\mathcal{CM}_\eta}=\sup_{I\subseteq\bt}\frac{\mu(S(I))}{|I|^\eta}
<\infty
$$
where
$$
S(I)=\{z=re^{i\theta}\in\mathbb D:\ 1-|I|\le r<1\ \ \&\ \
e^{i\theta}\in I\}
$$
is the Carleson square based on a subarc $I\subseteq\bt$.

For $a,b>0$ we define an integral operator $T_{a,b}$ as
\begin{equation*}
T_{a,b}f(z)=\int_\bd\frac{(1-|w|^2)^{b-1}}{|1-\bar w z|^{a+b}
}f(w)\rd A(w)\quad\forall\quad z\in\bd.
\end{equation*}

Below is a re-statement of Lemma 3.1.2 in \cite{X1}.

\begin{lem}\label{l31} Let $\eta\in(0,2)$, $a>\frac{2-\eta}{2}$,
$b>\frac{1+\eta}{2}$ and $f$ be Lebesgue measurable on $\bd$. If
$|f(z)|^2(1-|z|^2)^{\eta}\rd A(z)$ belongs to $\cm_\eta$, then
$|T_{a,b}f(z)|^2(1-|z|^2)^{\eta+2a-2}\rd A(z)$ also belongs to
$\cm_\eta$.
\end{lem}

\subsection{Distance estimates}\label{s32}

Importantly, Theorem \ref{t22} (ii) and the mean value inequality
for the subharmonic functions (see also the treatment for Case 3:
$1<\eta< 1+p$ in the proof of Theorem \ref{t22}) are employed to
derive the following optimal inclusion:
$$
0<\eta\le 1+p\ \ \&\ \ 1\le
p<\infty\Longrightarrow\mathcal{AL}_{p,\eta}\subseteq\mathcal{B}_{(p+1-\eta)/p}.
$$
This embedding, along with the P. Jone's distance estimation from the Bloch
functions to $BMOA$ (cf. \cite{GZ, Zh, X1}), suggests us to consider
the distance of an   $\mathcal{B}_{(p+1-\eta)/{p}}$ function $f$
to $\mathcal{AL}_{p,\eta}$. Such a distance is defined by

\begin{equation*}
\mathrm{dist}_{\cb_{(p+1-\eta)/p}}(f,\al_{p,\eta})
=\inf_{g\in\al_{p,\eta}}\|f-g\|_{\cb_{(p+1-\eta)/{p}}}\quad\forall\quad
f\in \mathcal{B}_{(p+1-\eta)/{p}}.
\end{equation*}

Motivated by Theorem 3.1.3 in \cite{X1}, we get the following
characterization of the above-defined distance, approaching the second
part of Theorem \ref{t11} (ii).

\begin{thm}\label{t31} For $\e >0$, $\eta\in(0,2)$ and $f\in\cb_{{(3-\eta)}/{2}}$, let
$$
\Omega_\e (f)=\{z\in\bd: (1-|z|^2)^{{(3-\eta)}/{2}}|f'(z)|\ge\e\}.
$$
Then
\begin{equation*}
\mathrm{dist}_{\cb_{{(3-\eta)}/{2}}}(f,\al_{2,\eta}) =\inf\Big\{\e>0:
\frac{\chi_{\Omega_\e(f)}(z)}{(1-|z|^2)^{2-\eta}}\rd A(z)\ \hbox{is\ in}\ \cm_\eta\Big\}.
\end{equation*}
\end{thm}
\begin{proof} For $f\in \cb_{{(3-\eta)}/{2}}$ and $z\in\mathbb D$, one has
the following representation formula (see, for example, \cite[(1.1)]{ABP}
or \cite[p. 55]{X1}):
$$
f(z)=f(0)+C\int_\bd\frac{(1-|w|^2)^{{(3-\eta)}/{2}}f'(w)}{\bar w
(1-\bar w z)^{1+{(3-\eta)}/{2}}}\rd A(w)=:f_1(z)+f_2(z),
$$
where
$$
\begin{cases}
f_1(z)=f(0)+C\int_{\Omega_\e(f)}\frac{(1-|w|^2)^{{(3-\eta)}/{2}}f'(w)}{\bar
w (1-\bar w z)^{1+{(3-\eta)}/{2}}}\rd A(w);\\
f_2(z)=C\int_{\bd\setminus\Omega_\e(f)}\frac{(1-|w|^2)^{{(3-\eta)}/{2}}f'(w)}{\bar
w (1-\bar w z)^{1+{(3-\eta)}/{2}}}\rd A(w);
\end{cases}
$$
with $C$ being a constant depending only on $\eta$. Then
\begin{align*}
&|f'_1(z)|\lesssim \|f\|_{\cb_{(3-\eta)/2}}
\int_{\bd}\frac{\chi_{\Omega_\e(f)}(w)}{ |1-\bar w
z|^{2+{(3-\eta)}/{2}}}\rd A(w)\\
&\approx\|f\|_{\cb_{(3-\eta)/2}} \int_{\bd}\frac{(1-|w|^2)}{ |1-\bar
w z|^{2+{(3-\eta)}/{2}}}\left(\frac{\chi_{\Omega_\e(f)}(w)}{1-|w|^2}\right)\rd
A(w)
\end{align*}
So, if
$$
\big(\chi_{\Omega_\e(f)}(z)\big)^2(1-|z|^2)^{\eta-2} \rd A(z)
$$
is in $\cm_\eta$, Lemma \ref{l31} (with $a=(3-\eta)/2$ and $b=2$) implies that
$$
|f'_1(z)|^2(1-|z|^2)\rd A(z)
$$
belongs to $\cm_{\eta}$. This means $f_1\in\al_{2,\eta}$. Meanwhile,
we have
$$
|f'_2(z)|\lesssim \e\int_{\bd}{|1-\bar w
z|^{{(\eta-3)}/{2}-2}}\rd
A(w)\lesssim\frac{\e}{(1-|w|^2)^{(3-\eta)/2}}.
$$
This gives
\begin{equation*}
\mathrm{dist}_{\cb_{{(3-\eta)}/{2}}}(f,\al_{2,\eta})
\le\inf\{\e>0: \chi_{\Omega_\e(f)}(z)(1-|z|^2)^{\eta-2} \rd A(z)
\in\cm_\eta\}.
\end{equation*}

In order to prove that the last inequality is actually an equality,
we may assume that $\e_0$ equals the right-hand quantity of the last
inequality and
\begin{align*}
\mathrm{dist}_{\cb_{{(3-\eta)}/{2}}}(f,\al_{2,\eta})<\e_0.
\end{align*}
It is enough to consider the case $\e_0>0$. Under $\e_0>0$ there exists an $\e_1$ such
that
$$
0<\e_1<\e_0\ \ \&\ \ \db(f,\al_{2,\eta})<\e_1.
$$
Hence, by definition, we can find a function $g\in\al_{2,\eta}$ such
that
$$
\|f-g\|_{\cb_{(3-\eta)/2}}<\e_1.
$$
Now for any $\e\in(\e_1,\e_0)$ we have that
$$
\chi_{\Omega_\e(f)}(z)(1-|z|^2)^{\eta-2} \rd A(z)
$$
is not in $\cm_\eta$. But, $\|f-g\|_{\cb_{(3-\eta)/2}}<\e_1$ yields
$$
(1-|z|^2)^{(3-\eta)/2}|g'(z)|>(1-|z|^2)^{(3-\eta)/2}|f'(z)|-\e_1
\quad\forall z\in\bd,
$$
and so
\begin{align*}
\chi_{\Omega_\e(f)}(z)\le \chi_{\Omega_{\e-\e_1}(g)}(z)\quad
\forall\quad z\in\bd.
\end{align*}
 This implies that
\begin{align*}
\chi_{\Omega_{\e-\e_1}(g)}(z)(1-|z|^2)^{\eta-2} \rd A(z)
\end{align*}
does not belong to $\cm_\eta$. However,
\begin{align*}
&\chi_{\Omega_{\e-\e_1}(g)}(z)(1-|z|^2)^{\eta-2} \rd A(z)\\
&\le\chi_{\Omega_{\e-\e_1}(g)}(z)\frac{(1-|z|^2)}{(1-|z|^2)^{3-\eta}}
\rd A(z)\\
&\le\frac{|g'(z)|^2}{(\e-\e_1)^2}(1-|z|^2)\chi_{\Omega_{\e-\e_1}(g)}(z)\rd
A(z)\\
&\le{(\e-\e_1)^{-2}}|g'(z)|^2(1-|z|^2)\rd A(z).
\end{align*}
Since $g\in\al_{2,\eta}$ means that
$$
|g'(z)|^2(1-|z|^2)\rd A(z)
$$
is in $\cm_\eta$, and consequently
$$
\chi_{\Omega_{\e-\e_1}(g)}(z)(1-|z|^2)^{\eta-2} \rd A(z)
$$
is in $\cm_\eta$. Now, a contradiction occurs. Thus we must have
$$
\db(f,\al_{2,\eta})=\e_0
$$
as required.
\end{proof}

\begin{rem}\label{r31}
\item{\rm(i)} Theorem \ref{t31} characterizes the closure of $\al_{2,\eta}$ in the
$\cb_{{(3-\eta)}/{2}}$ norm. That is, for
$f\in\cb_{{(3-\eta)}/{2}}$, $f$ is in the closure of $\al_{2,\eta}$
in the $\cb_{{(3-\eta)}/{2}}$ norm if and only if for every $\e>0$,
\begin{align*}
\int_{\Omega_\e (f)\cap S(I)}(1-|z|^2)^{\eta-2}\rd A(z)\lesssim
|I|^\eta
\end{align*}
holds for any Carleson square $S(I)\subseteq\mathbb D$.

\item{\rm(ii)} The proof of Theorem \ref{t31} depends on an important fact that
$f\in\al_{2,\eta}$ if and only if $|f'(z)|^2(1-|z|^2)\rd
A(z)$ is a member of $\cm_\eta$. Given $1<p<\infty$, if we define the analytic
function space $\ca_{\pe}$ to be the space of all $H^p$ functions
satisfying
\begin{align*}
\|f\|_{\ca_{\pe}}^p=\big\| |f'(z)|^p (1-|z|^2 )^{p-1}\rd
A(z)\big\|_{\cm_\eta}<\infty,
\end{align*}
then $\ca_{2,\eta}=\al_{2,\eta}$. This $\ca_{\pe}$ is closely
related to $\al_{\pe}$. It can also be proved that
$\ca_{\pe}\subseteq\cb_{1+(1-\eta)/p}$. In a future article, we
will characterize the distance of a function in
$\cb_{1+(1-\eta)/p}$ to $\ca_{\pe}$ under $0<\eta\le 1+p$ and $1\le
p<\infty$.
\end{rem}

\section{Compositions for analytic Campanato spaces}\label{s4}
\setcounter{equation}{0}

\subsection{Actions between analytic Campanato spaces}\label{s41}

In what follows, for $0<r<1$ and an analytic self-map $\p $ of
$\mathbb D$ let
$$
N_r(\p
,w)=\sum_{z\in\p^{-1}\{w\}}\log^+(\frac{r}{|z|})\quad\forall\quad
w\in\mathbb D,
$$
where $\log^+x=\max\{\log x, 0\}$, and then set
$$
N(\p,w)=\lim_{r\to 1}N_r(\p,w)\quad\forall\quad w\in\mathbb D
$$
be the Nevanlinna counting function of $\phi$. The importance of such a counting function in the study of compositions on $H^p$ (cf. \cite{CM, JMCP, JS, SM})  initially comes from the following fundamental formula (cf. \cite[p. 2336]{S}).

\begin{lem}\label{l41} Let $\p $ be an analytic self-map of $\mathbb D$ and $f\in H^p$ with $p\in [1,\infty)$. Then
\begin{equation*}
\|f\circ\p\|^p_{p}=|f(\p(0))|^p+\frac{p^2}{2}\int_\bd |f(w)|^{p-2}
|f'(w)|^2 N(\p,w)\,\rd A(w).
\end{equation*}
In particular,
\begin{equation*}
\|f\|^p_{p}=|f(0)|^p+\frac{p^2}{2}\int_\bd |f(w)|^{p-2} |f'(w)|^2
\log\frac1{|w|}\,\rd A(w).
\end{equation*}
\end{lem}
This is used to establish two precise estimates for $\|\p \|_p$
extending \cite[(2.8)]{L}.

\begin{lem}\label{l42} If $\p$ is an analytic self-map of $\bd$ with $\p(0)=0$,
then

\item{\rm(i)}
$$
\|\p\|_{2}^2=2\int_{\mathbb D}N(\p,w)\rd A(w).
$$
\item{\rm(ii)}
\begin{equation*}
N(\p,z)\le \frac{4}{\log
2}\|\p\|^2_{2}\log\frac1{|z|}\quad\forall\quad
z\in\bd\setminus\frac{1}{2}\bd.
\end{equation*}
\end{lem}
\begin{proof} (i) This follows from taking $f(z)=z$ in Lemma \ref{l41}.

(ii) Given $w\in\mathbb D\setminus\{0\}$. Using (i) and the submean
value property of $N(\p ,\cdot)$, we get
$$
N(\p ,w)\le \frac{1}{(1-|w|)^2}\int_{|z-w|<1-|w|}N(\p,z)\rd A(z)\le
\frac{\|\p \|_p^p}{(1-|w|)^2}.
$$
Consequently,
$$
N_r(\p,w)\le N(\p,w)\le 4\|\p \|_{2}^2\quad\hbox{for}\quad
|w|=1/2.
$$
Note that the well-known Littlewood's inequality ensures that
$N_r(\p ,w)$ approaches zero uniformly as $|w|\to 1$. Note also that
$N_r(\p ,w)$ is subharmonic. So it is bounded above by the harmonic
function on the annulus $\{w\in\mathbb D: 1/2<|w|<1\}$ having these boundary
values. Hence,
$$
N(\p ,w)=\lim_{r\to 1}N_r(\p ,w)\le \frac{4}{\log 2}\|\p \|_{2}^2
\log \frac1{|w|}\quad\hbox{for}\quad 1/2<|w|<1.
$$
\end{proof}

As the generalization of an $H^2$ composition result in \cite{L},
the following splitting inequality essentially improves the well-known sub-ordination principle for $H^p$ with $p\ge 2$.

\begin{thm}\label{t41} Let $2\le p<\infty$. Then
\begin{equation*}
\|f\circ\p\|_{p}\lesssim\|f\|_{p}\|\p\|^{2/p}_{p}
\end{equation*}
holds for all $f\in H^p$ and analytic self-maps $\p$ of  $\,
\mathbb{D}$ with $f(0)=\p(0)=0$.
\end{thm}

\begin{proof} According to Lemma \ref{l41}, we have
$$
\|f\circ\p\|^p_{p}=\frac{p^2}{2}\int_\bd |f(w)|^{p-2} |f'(w)|^2
N(\p,w)\,\rd
A(w)=:\int_{\frac{1}{2}\bd}+\int_{\bd\setminus\frac{1}{2}\bd}.
$$
Using the Cauchy integral formula and H\"{o}lder's inequality, we
get
$$
|f(z)|+(1-|z|)|f'(z)|\lesssim{\|f\|_{p}}(1-|z|)^{-\frac1p}\quad\forall\quad
z\in\mathbb D.
$$
Thus by Lemma \ref{l42} (i) and $p\ge 2$, we find
\begin{align*}
&\int_{\frac{1}{2}\bd}=\frac{p^2}{2}\int_{\frac{1}{2}\bd}
|f(w)|^{p-2}
|f'(w)|^2 N(\p,w)\,\rd A(w)\\
&\lesssim\|f\|_{p}^p\int_{\frac{1}{2}\bd}
N(\p,w){(1-|w|)^{-p-2}}\,\rd A(w)\\
&\lesssim\|f\|_{p}^p\int_{\frac{1}{2}\bd} N(\p,w)\,\rd
A(w)\\
&\lesssim \|\p\|^2_{2}\|f\|_{p}^p
\end{align*}
Meanwhile, Lemma \ref{l42}(ii) implies
\begin{align*}
&\int_{\bd\setminus\frac{1}{2}\bd}=\frac{p^2}{2}\int_{\bd\setminus\frac{1}{2}\bd}
|f(w)|^{p-2}
|f'(w)|^2 N(\p,w)\,\rd A(w)\\
&\lesssim\|\p\|^2_{2}\int_{\bd\setminus\frac{1}{2}\bd}
|f(w)|^{p-2} |f'(w)|^2 \log\frac1{|w|}\,\rd A(w)\\
&\lesssim\|\p\|^2_{2}\|f\|_{p}^p\, .
\end{align*}
Now, since $p\ge 2$ the above estimates and the H\"older inequality are utilized to derive
$$
\|f\circ\p\|^p_{p}\lesssim\|\p\|^2_{2}\|f\|_{p}^p\lesssim\|\p\|^2_{p}\|f\|_{p}^p,
$$
as desired.
\end{proof}

We have the following assertion which covers Theorem \ref{t11} (iii)
and implies that $\al_{\pe}$ embeds into $\al_{q,\lambda}$ under
$p\ge q=2$ and $(1-\lambda)p\ge(1-\eta)q$.

\begin{thm}\label{t42} Let $0<\eta,\lambda<2=q\le p<\infty$ and $\p $ be an analytic self-map of $\mathbb D$. Then
$$
\|\cp f\|_{\mathcal{AL}_{q,\lambda,\ast}}\lesssim\|f\|_{\mathcal{AL}_{p,\eta,\ast}}\quad\forall\quad f\in\al_{p,\eta}
$$
when and only when
$$
\sup_{w\in
\bd}\frac{(1-|w|^2)^{(1-\lambda)/q}}{(1-|\p(w)|^2)^{(1-\eta)/p}}
\|\sigma_{\p(w)}\circ\p\circ\sw\|_{q}<\infty.
$$
\end{thm}

\begin{proof}  If $f\in \al_{\pe}$ and $2=q\le p$, then an application of Theorem \ref{t41}, plus H\"older's inequality, derives
\begin{align*}
&\|\cp(f)\|^q_{\mathcal{AL}_{q,\lambda,*}}\approx\sup_{w\in\bd}(1-|w|^2)^{1-\lambda}\|f\circ\p\circ\sw-f(\p(w))\|^q_{q}\\
&\approx\sup_{w\in\bd}(1-|w|^2)^{1-\lambda}\|(f\circ\sigma_{\p(w)}-f(\p(w)))
\circ(\sigma_{\p(w)}\circ\p\circ\sw)\|^q_{q}\\
&\lesssim\sup_{w\in\bd}(1-|w|^2)^{1-\lambda}\|f\circ\sigma_{\p(w)}-f(\p(w))\|^q_{p}
\|\sigma_{\p(w)}\circ\p\circ\sw\|^q_{q}\\
&\lesssim \sup_{w\in
\bd}\frac{(1-|w|^2)^{1-\lambda}}{(1-|\p(w)|^2)^{q(1-\eta)/p}}
\|\sigma_{\p(w)}\circ\p\circ\sw\|^q_{q}\|f\|^q_{\pe,*}
\end{align*}
and hence the forward implication holds.

To verify the backward implication, for $b,z\in\mathbb D$ let
$f_b(z)=\frac{1}{1-\bar{b}z}$. Then

\begin{align*}
&\|f_b\|^p_{\mathcal{AL}_{p,\eta,*}}\approx\sup_{w\in\bd}(1-|w|^2)^{1-\eta}\|f_b\circ\sw-f_b(w)\|^p_{p}\\
&\approx\sup_{w\in\bd}(1-|w|^2)^{1-\eta}\int_\bt|f_b(\zeta)-f_b(w)|^p\frac{1-|w|^2}{|\zeta-w|^2}\,|\rd
\zeta|\\
&\approx\sup_{w\in\bd}(1-|w|^2)^{1-\eta}\int_\bt\left|
\frac{\bar{b}(\zeta-w)}{(1-\bar{b}\zeta)(1-\bar{b}w)}\right|^p\frac{1-|w|^2}{|\zeta-w|^2}\,|\rd
\zeta|\\
&\approx\sup_{w\in\bd}\frac{|b|^p(1-|w|^2)^{1-\eta}}{|1-\bar{b}w|^p}\int_\bt
\left|\frac{\zeta-w}{1-\bar{b}\zeta}\right|^p\frac{1-|w|^2}{|\zeta-w|^2}\,|\rd \zeta|\\
&\approx\sup_{w\in\bd}\frac{|b|^p(1-|w|^2)^{1-\eta}}{|1-\bar{b}w|^p}\int_\bt
\left|\frac{w-\sw(\zeta)}{1-\bar{b}\sw(\zeta)}\right|^p\,|\rd \zeta|\\
&\approx\sup_{w\in\bd}\frac{|b|^p(1-|w|^2)^{1-\eta}}{|1-\bar{b}w|^p}\int_\bt
\left|\frac{1-|w|^2}{1-\bar b w-\bar{w}\zeta+\bar b\zeta}\right|^p\,|\rd \zeta|\\
&\approx\sup_{w\in\bd}\frac{|b|^p(1-|w|^2)^{p+1-\eta}}{|1-\bar{b}w|^{2p}}\int_\bt
\left|\frac{1}{1-\sigma_{\bar b}(\bar w)\zeta}\right|^p\,|\rd \zeta|\\
&\lesssim\sup_{w\in\bd}\frac{|b|^p(1-|w|^2)^{p+1-\eta}}{|1-\bar{b}w|^{2p}}
{(1-|\sigma_{\bar b}(\bar w)|)^{1-p}}\\
 &\approx\sup_{w\in\bd}\frac{|b|^p(1-|w|^2)^{2-\eta}}{|1-\bar{b}w|^{2}(1-|b|^2)^{p-1}}\\
\end{align*}
If
$$
F_b(z)={(1-|b|^2)^\frac{p+\eta-1}{p}}/{(1-\bar b
z)}\quad\forall\quad z\in\mathbb D,
$$
then
$$
\|F_b\|^p_{\mathcal{AL}_{p,\eta,*}}\lesssim
\sup_{w\in\bd}\frac{|b|^p(1-|w|^2)^{2-\eta}(1-|b|^2)^{\eta}}{|1-\bar{b}w|^{2}
}\lesssim 1,
$$
and hence $F_b$ is a bounded function in $\al_{\pe}$ with a bound being independent of $b$.

Now, let
$$
\|C_\p f\|_{\mathcal{AL}_{q,\lambda,\ast}}\lesssim\|f\|_{\mathcal{AL}_{p,\eta,\ast}}\quad\forall\quad f\in\mathcal{AL}_{p,\eta}.
$$
Then for $\tau:=(p+\eta-1)/p$ one has
\begin{align*}
&1\gtrsim\|F_b\|^{q}_{\mathcal{AL}_{p,\eta,\ast}}\\
&\gtrsim\|C_\p F_b\|^q_{\mathcal{AL}_{q,\lambda,\ast}}\\
&\approx\sup_{b,w\in\bd}(1-|w|^2)^{1-\lambda}\int_\bt
|F_b\circ\p\circ\sw-F_b(\p(w))|^q\,|\rd\zeta|\\
&\approx\sup_{b,w\in\bd}(1-|w|^2)^{1-\lambda}
\int_\bt |F_b\circ\p-F_b(\p(w))|^q |\sigma'_w(\zeta)|\,|\rd\zeta|\\
&\approx\sup_{b,w\in\bd}(1-|w|^2)^{1-\lambda}\int_\bt
\left|\frac{(1-|b|^2)^\tau}{(1-\bar b \p(\zeta))}-
\frac{(1-|b|^2)^\tau}{(1-\bar b \p(w))}\right|^q |\sigma'_w(\zeta)|\,|\rd\zeta|\\
&\approx\sup_{b,w\in\bd}\frac{|b|^q(1-|w|^2)^{1-\lambda}(1-|b|^2)^{\tau q}}{|1-\bar b \p(w)|^q}\int_\bt \left|\frac{\p(\zeta)-\p(w)}{(1-\bar b \p(\zeta))} \right|^q
|\sigma'_w(\zeta)|\,|\rd\zeta|\\
&\gtrsim\sup_{w\in\bd}\frac{|\p(w)|^q(1-|w|^2)^{1-\lambda}}{(1-| \p(w)|^2)^{q(1-\tau)}}\int_\bt \left|\sigma_{\p(w)}\circ\p \right|^q|\sigma'_w(\zeta)|\,|\rd\zeta|\\
&\approx\sup_{w\in\bd}\frac{|\p(w)|^q(1-|w|^2)^{1-\lambda}}{(1-|\p(w)|^2)^{q(1-\tau)}}\|\sigma_{\p(w)}\circ\p\circ\sw\|^q_{q},
\end{align*}
and hence the desired implication follows.

\end{proof}

\subsection{Actions between analytic Campanato and Bloch type spaces}\label{s42} The
previous discussion leads to a consideration of the actions of $C_\p
$ sending $\mathcal{AL}_{p,\eta}$ to $\mathcal{B}_\alpha$ and its
converse. To do so, we need the following existence result.

\begin{lem}\label{l43} For $0<\alpha<\infty$, there are two functions $f_1,f_2$ in
$\mathcal{B}_\alpha$ such that
$$
(1-|z|^2)^{2\alpha}\big(|f_1'(z)|^2+|f_2'(z)|^2\big)\approx
1\quad\forall\quad z\in\mathbb D.
$$
\end{lem}
\begin{proof} This follows immediately from the inequalities (2.2) and (2.4) in the argument for \cite[Theorem 2.1.1]{X0}.
\end{proof}

Below is an alternative to Theorem \ref{t11} (iv).

\begin{thm}\label{t43} Let $0<\eta<1+p<\infty$, $0<\alpha, p-1<\infty$ and $\p $ be an analytic self-map of $\mathbb D$.

\item{\rm(i)}
$$
\|C_\p f\|_{\mathcal{B}_\alpha}\lesssim\|f\|_{\mathcal{AL}_{p,\eta,\star}}\quad\forall\quad f\in \mathcal{AL}_{p,\eta}
$$
if and only if
$$
\sup_{w\in\mathbb D}\frac{(1-|w|^2)^\alpha|\p '(w)|}{(1-|\p
(w)|^2)^\frac{p+1-\eta}{p}}<\infty.
$$

\item{\rm(ii)}
$$
\|C_\p f\|_{\mathcal{AL}_{p,\eta,\star}}\lesssim\|f\|_{\mathcal{B}_\alpha}\quad\forall\quad f\in \mathcal{B}_\alpha
$$
if and only if
$$
\sup_{I\subseteq\bt}\left( \frac1{|I|^\eta} \int_I
\left(\int_{1-|I|}^1 |\p '(r \zeta)|^2 (1-r)^{1-2\alpha}\rd
r\right)^\frac{p}{2}\,\frac{|\rd\zeta|}{2\pi}\right)^\frac{1}{p}<\infty.
$$
\end{thm}

\begin{proof} (i) The growth of functions in $\mathcal{AL}_{p,\eta}$ presented in the beginning of Section \ref{s3} derives that if $f\in \mathcal{AL}_{p,\eta}$ then
$$
\sup_{w\in\mathbb D}{(1-|w|^2)^\alpha|(C_\p
f)'(w)|}\lesssim\|f\|_{\mathcal{AL}_{p,\eta,\star}}\sup_{w\in\mathbb
D}\frac{(1-|w|^2)^\alpha|\p '(w)|}{(1-|\p
(w)|^2)^\frac{p+1-\eta}{p}}
$$
and hence the sufficiency of (i) is true. The necessity of (i) just
follows from a simple calculation with a given point $w\in\mathbb D$
and the test function in $\mathcal{AL}_{p,\eta}$ below:
$$
f_w(z)=\frac{(1-|\p (w)|^2)^{(p+\eta-1)/p}}{1-\overline{\p
(w)}z}\quad\forall\quad z\in\mathbb D.
$$

(ii) The sufficiency can be checked by using the definition of
$\mathcal{B}_\alpha$ and Theorem \ref{t22} (ii). Concerning the
necessity, we utilize those two $\mathcal{B}_\alpha$ functions
$f_1,f_2$ in Lemma \ref{l43} and the elementary estimate
$$
U^{p/2}+V^{p/2}\approx (U+V)^{p/2}\quad\forall\quad U,V\ge 0
$$
to get that if
$$
\|C_\p f\|_{\mathcal{AL}_{p,\eta,\star}}\lesssim\|f\|_{\mathcal{B}_\alpha}\quad\forall\quad f\in \mathcal{B}_\alpha
$$
then
\begin{align*}
&\infty>\|f_1\|_{\mathcal{B}_\alpha}^p+\|f_2\|_{\mathcal{B}_\alpha}^p\\
&\gtrsim\|C_\p  f_1\|_{\mathcal{AL}_{p,\eta,\star}}^p+\|C_\p  f_2\|_{\mathcal{AL}_{p,\eta,\star}}^p\\
&\gtrsim |I|^{-\eta}\int_I\Big(\int_{1-|I|}^1|(f_1\circ\p )'(r\zeta)|^2(1-r)\rd r\Big)^\frac{p}{2}\,|\rd\zeta|\\
&\ \ +|I|^{-\eta}\int_I\Big(\int_{1-|I|}^1|(f_2\circ\p )'(r\zeta)|^2(1-r)\,\rd r\Big)^\frac{p}{2}\,|\rd\zeta|\\
&\gtrsim |I|^{-\eta}\int_I\Big(\int_{1-|I|}^1 \big(|(f_1\circ\p )'(r\zeta)|^2+|(f_2\circ\p )'(r\zeta)|^2\big)(1-r)\,\rd r\Big)^\frac{p}{2}\,|\rd\zeta|\\
&\gtrsim{|I|^{-\eta}} \int_I\left(\int_{1-|I|}^1 |\p '(r \zeta)|^2
(1-r)^{1-2\alpha}\rd r\right)^\frac{p}{2}\,{|\rd\zeta|}
\end{align*}
holds any subarc $I\subseteq\bt$. This implies the desired
inequality.
\end{proof}

\end{document}